\documentclass[letterpaper,conference]{ieeeconf}

\IEEEoverridecommandlockouts
\usepackage{amssymb,amsmath}
\usepackage{graphicx,color}
\usepackage{algorithm,algorithmic}
\usepackage{bbm}
\usepackage{tikz}
\usepackage{url}
\usepackage[compress]{cite}
\usepackage{todonotes}

\newtheorem{problem}{Problem}
\newtheorem{theorem}{Theorem}

\title{ Preserving Privacy of Agents in Participatory-Sensing Schemes\\  for Traffic Estimation\vspace{-.1in} \thanks{The work of the authors was supported by a McKenzie Fellowship, ARC grant LP130100605, a grant from Melbourne School of Engineering.}}
\author{Farhad Farokhi and Iman Shames\vspace{-.1in}\thanks{The authors are with the Department of Electrical and Electronic Engineering, the University Melbourne, Parkville, VIC 3010, Australia.}}

\begin{document}
\maketitle

\begin{abstract} A measure of privacy infringement for agents (or participants) {  travelling across a transportation network} in participatory-sensing schemes for traffic estimation is introduced. The measure is defined to be the conditional probability that an external observer assigns to the {  private nodes} {  in the transportation network}, e.g., location of home or office, given all the position measurements that it broadcasts over time. An algorithm for finding an optimal trade-off between the measure of privacy infringement and the expected estimation error, captured by the number of the nodes over which the participant stops broadcasting its position, is proposed. The algorithm searches over a family of policies in which an agent stops transmitting its position measurements if its distance (in terms of the number of hops) to the privacy sensitive node is smaller than a {  prescribed threshold}. Employing such symmetric policies are advantageous in terms of the resources required for implementation and the ease of computation. The results are expanded to more general policies. Further, the effect of the heterogeneity of the population density on the optimal policy is explored. Finally, the relationship between the betweenness measure of centrality and the optimal privacy-preserving policy of the agents is numerically explored.
\end{abstract}

\section{Introduction}
A sharp rise in the number of the networked platforms, such as smart phones and wearable gadgets, has enabled new technologies, such as participatory-sensing schemes, to be commercially viable. In these schemes, agents (or participants)  and their networked devices act as sensing units to estimate a variable of interest. Waze\footnote{\url{https://www.waze.com/}} and Mobile Millennium\footnote{\url{http://traffic.berkeley.edu/}} can be mentioned as examples of commercial and academic products that use participatory-sensing schemes for measuring  the traffic flow in real time. These systems often recruit agents that are willing to provide data (by directly providing reports or letting the sensors on their devices to be remotely used). However, gathering data usually reveals some private information about the agents (e.g., the location of their house or office among many other variables), which might make them opt out of the system or switch off their connected devices (more often than needed). One way to alleviate the privacy related anxieties of the agents is to ``systematically corrupt'' the measurements collected by them. The objective of such corruption is to obfuscate the private information of each agent. A direct link between the intensity of corruption and the quality of the estimate, on the one hand, and the availability of the private information, on the other hand, can be established. This observation is at the core of the literature on differential privacy (see~\cite{Dwork2008,chen2012differentially, kargl2013differential,wang2014entropy, le2012differentially} among other studies), where the responses to statistical queries on random databases are typically corrupted by Laplace noise to protect the privacy of the individuals in the database.

In transportation systems, differential privacy might not be well-suited for preserving the privacy of the agents. To demonstrate this, as an example, consider a scenario in which the networked devices of the agents in a participatory-sensing scheme provide Global Positioning System (GPS) measurements of their position in equidistant intervals of time. From the perspective of a subscribing agents, the privacy infringement is only an issue if the agent is close to a privacy sensitive node\footnote{A rational agent, which is privacy-conscious enough to not want its path to be recorded at all, simply does not participate in the scheme.}, such as its home or office. Therefore, when travelling through most of the network, transmitting accurate measurements of its position is not infringing the privacy of the agent (and thus the addition of the noise is conservative). In addition, if an agent is staying in a certain node for long durations of time (e.g., over night at home or during work hours at office), a simple averaging can significantly reduce the effect of any additive noise on the position measurements and, thus, revealing the position of {  private nodes}. Hence, adding noise to the sensor measurements (at least indiscriminately with respect to the position) to protect the privacy of the agent is neither necessary nor beneficial. Therefore, a different approach is required for protecting the privacy of the agents in which they completely stop broadcasting their position based on their proximity to the {  private nodes}. Such a method can form the basis of a software package that gives rise to the democratization of privacy-preserving tools and measures in transportation networks. The software simply asks for the privacy requirements of the agents and automatically turns off their GPS units if they are closer than an optimally selected level to their {  private nodes}. {  Note that this proposal is not the novel aspect of this paper. Many commercial softwares provide similar (but application-specific) solutions. For instance, Strava\footnote{\url{http://www.strava.com/}} enables users (mainly runners and cyclists) to construct a privacy zone (by specifying a position and radius) in which the GPS location of the user is not reported and is thus not revealed to a third-party\footnote{Thieves have used such measurements to steal valuable bikes. \url{http://www.telegraph.co.uk/technology/news/11372189/Cycling-apps-put-you-at-risk-from-hi-tech-burglars.html}}. Though undoubtedly powerful in preserving privacy of the users over one single trip, such zones might prove ineffective in the long run. For instance, imagine that the same user gets out of its privacy zone in three different locations (on three different occasions). Now, equipped with the knowledge that the user's house  is in the same distance from these three positions (without even knowing that distance), a well-resourced attacker can simply pinpoint the position of the house of the user by triangulation. Therefore, more careful design strategies are required to protect the privacy of the users. This paper provides a \textit{provably privacy-preserving} method for constructing the privacy zones based on the underlying transportation network.}

In this paper, first, a measure of privacy infringement for the agents is introduced. The measure is equal to the conditional probability that an external observer assigns to the {  private nodes} of the agent given all its position measurements over time. By carefully selecting the set of nodes and edges over which the agent does not transmit its position, it can effectively reduce this conditional probability to mislead the maximum likelihood filters that can be used to find the location of the {  private nodes}. However, this can only be achieved at the cost of not reporting the position at some nodes in the network, which would reduce the quality of the estimation provided by the participatory-sensing scheme. An algorithm to find an optimal trade-off between these two competing interests is provided. The algorithm searches over a family of policies in which the agent stops broadcasting its position measurements if its distance in terms of the number of hops to the {  private nodes} is smaller than a level. Employing such symmetric policies are advantageous as they are easy to compute and to implement. This is due to the fact that the agent only needs to count the number of the nodes on its path and does not require a map of the city. However, they have a drawback because an agent that uses a symmetric policy should exclude a larger subgraph from its measurements to achieve the same level of privacy as an agent that does not restrict itself to such policies. The results are subsequently expanded to more general (possibly asymmetric) policies. The effect of the population density is then explored. The impact of the betweenness measure of centrality of the {  private nodes} of an agent on the policy of that agent is numerically explored on a random geometric graph. The betweenness measure of centrality for a node is defined as the number of the shortest paths in the graph that pass through that node divided by the total number of the shortest paths~\cite[p.\,39]{jackson2010social}. Intuitively, a node with a high betweenness measure connects many nodes to each other (and is thus very central). However, a node with a low betweenness measure can be removed from the graph with no devastating consequences. It is observed numerically that, for the node with the maximum betweenness in a random geometric graph, a fairly large exclusion radius should be selected to reduce the measure of the privacy infringement. In addition, by slightly increasing the radius of the exclusion zone, the number of the nodes and edges over which the agent stop its broadcasting its position measurements rapidly increases. The opposite behaviour is witnessed for the node with the minimum betweenness. Finally, the results are demonstrated on the City of Melbourne (the center of the Melbourne metropolitan area).

The rest of the paper is organized as follows. This section is concluded by presenting some useful notations. Section~\ref{sec:problem} introduces the privacy measure and mathematically formulates the problem. Section~\ref{sec:sym} provides an algorithm for finding a symmetric policy that balances between the need for privacy  and the quality of the estimation. Section~\ref{sec:asym} extend the results to asymmetric policies. The effect of the heterogeneity of the population density is explored in Section~\ref{sec:density}. The numerical examples are illustrated and their implications are discussed in Section~\ref{sec:numerical}. Finally, the paper is concluded in Section~\ref{sec:conclusions}.

\subsection{Notations}
Consider an undirected $\mathcal{G}=(\mathcal{V},\mathcal{E})$, where  $\mathcal{V}$ is the vertex set and $\mathcal{E}\subseteq \mathcal{V}\times \mathcal{V}$ is the edge set. An undirected edge between nodes $i,j\in\mathcal{V}$ is denoted by $\{i,j\}$.
A \emph{walk} is a sequence $v_{1},e_{1},v_{2}, e_{2},\ldots,e_{k-1},v_{k}$ of vertices $v_{i}\in\mathcal{G}$ and edges $e_{i}\in\mathcal{E}$ such
that $e_{i}=(v_{i},v_{i+1})$ for all $i<k$.  The vertices $v_{1}$ and $v_{k}$ are called the \emph{initial} and  the \emph{terminal} vertices of the walk, respectively. The \emph{length} of the walk is the number of edges in it.  A \emph{path} is a walk in which all vertices are distinct. The \emph{distance} $d(i,j)$  between two vertices $i,j\in\mathcal{V}$ is the length of a \emph{shortest} path with initial vertex $i$ and terminal vertex $j$. Additionally, define the set of nodes with a distance $\delta$ from node $i$ by $\mathcal{D}_i(\delta)$ where $\mathcal{D}_i(\delta)= \{j|j\in\mathcal{V}, d(i,j)=\delta\}$. The length of the longest shortest path between any two vertices of a graph is the graph's \emph{diameter}.

\section{Problem Formulation} \label{sec:problem}
A {\em{transportation network}}, or simply a network, is modelled by an undirected graph $\mathcal{G}=(\mathcal{V},\mathcal{E})$, { where the nodes and the edges represent the intersections and the road segments, respectively}. Consider a set of mobile agents, denoted by $A$, that traverse throughout the network along its edges. The agents are participating in a passive\footnote{The term passive refers to the fact that the agents are not required to report traffic incidents (e.g., as in Waze). The participatory-sensing scheme receives GPS measurements of the agents at different times to construct an estimate of the traffic flow. An example of such a system can be found in~\cite{Herrera2010568}. } participatory-sensing scheme to estimate the state of the traffic (i.e., the number of the vehicles\footnote{Note that there might be many more vehicles than agents since not everyone is participating in the sensing scheme.} on each edge) in the network. Each agent $a\in A$ is assumed to visit all nodes in $\mathcal{V}$ at least once\footnote{For this assumption to be satisfied, it is only required that the agents visit any node by a non-zero probability (no matter how small). Then, as time goes to infinity, the agents visit all nodes with probability one. Removing this assumption can make the observer's task at determining the private nodes of an agent harder. Therefore, this analysis can be used as a worst-case analysis of the privacy preserving policies for the agents. }.
Any agent $a\in A$ can choose whether to broadcast its position or not. In case of deciding not to broadcast the position, it can simply turn off the GPS sensor on its connected device or terminate the participatory-sensing application. Assume that agent $a$ does not want an external observer (with access to all its position measurements) to be able to determine if it has ever visited a prescribed set of nodes denoted by $\mathcal{S}_a\subseteq\mathcal{V}$. This set is assumed to be only known by agent $a$. The nodes $s_a\in \mathcal{S}_a$ are termed the \emph{{  private nodes}} for agent $a$. For the sake of the simplicity of exposition, assume that $\mathcal{S}_a=\{s_a\}$ for some $s_a\in\mathcal{V}$.
The results of this paper can be easily generalized  to the case that $\mathcal{S}_a$ is not a singleton (so long as the elements of $\mathcal{S}_a$ are ``far enough'' from each other on the graph).

Assume that the observer has a uniform prior on privacy sensitive node $s_a$ of agent $a$. This assumption is removed later in the paper to account for the heterogeneity of the population density. Given the positions that the agent has broadcast (and the time stamp of those positions), the observer can construct the conditional probability $p_a(v)$ that node $v$ belongs to the set $\mathcal{S}_a$. Agent $a\in A$ wants to keep the conditional probability $\pi_a:=p_a(s_a)$ for $s_a\in\mathcal{S}_a$ small. This way, the observer does not have a good chance for correctly inferring that the agent has visited the node $s_a$. Therefore, this conditional probability can be seen as a measure of privacy infringement for agent $a$. However, the agents want to achieve this goal with not drastically degrading the performance of the participatory-sensing scheme. Note that, if such a constraint is not enforced, the best policy of an agent is to never broadcast its position.

Throughout this paper, agent $a$ is assumed to follow a policy that instructs it to stop broadcasting its position when it is on a link $\{i,j\}$ if $d(i,s_a)\leq h_a$ or $d(j,s_a)\leq h_a$ for $s_a\in\mathcal{S}_a$, where $h_a\in\mathbb{N}$ is a prescribed integer number that is only known by agent $a$. The goal is to find $h_a$ optimally.  Such a policy is favoured (in comparison to more general asymmetric ones) as it can be easily implemented: it does not require a map of the city. The agent simply stops broadcasting when its hop-distance to its private node, $s_a$, is smaller than $h_a$.

As described earlier, the agents do not want to drastically degrade the quality of the participatory-sensing scheme.
Let $\mathcal{E}(s_a,h_a)$ denote the set of edges over which agent $a$ stop broadcasting its position. It can be shown that
\begin{align*}
\mathcal{E}(s_a,h_a)=\{\{i,j\}\in\mathcal{E}\,|\,d(i,s_a)\leq h_a \vee d(j,s_a)\leq h_a\}.
\end{align*}
This is inversely proportional to the quality of the estimation as, in many sensing schemes (e.g., the least mean square method), the covariance of the estimation error reduces by increasing the number of the measurements. Therefore, agent~$a$ may wish to keep $r_a:=|\mathcal{E}(s_a,h_a)|$ relatively small.

\begin{problem} \label{prob:1} For any agent $a\in A$, find $h_a$ to minimize $\pi_a+\gamma_a r_a$, where the constant $\gamma_a>0$ is inversely proportional to the importance of privacy for agent $a$.
\end{problem}

For very large constant $\gamma_a$, the cost of the agent behaves similarly to $r_a$. Therefore, the best policy of the agent is to report on all edges and, thus, select $h_a=0$. However, for very small $\gamma$, the cost of the agent is mostly determined by $\pi_a$. Thus, the best policy of the agent is  to never broadcast its position or, equivalently, set $h_a$ to be larger than the diameter of $\mathcal{G}$. Alternatively, the agent can solve the following problem. This problem formulation avoids the use of a non-intuitive parameter $\gamma_a$.

\begin{problem} \label{prob:2} For any agent $a\in A$, find $h_a$ to minimize $r_a$ subject to $\pi_a\leq \xi_a$, where the constant $\xi_a\in[0,1]$.
\end{problem}

If $\xi_a\leq 1/|\mathcal{V}|$, the solution of Problem~\ref{prob:2} is to select $h_a$ larger than the diameter of $\mathcal{G}$. This is because, upon providing any position measurements, the conditional probability $\pi_a$ can only be increased (i.e., only private information can be leaked by providing measurements). For $\xi_a=1$,  the best policy of the agent is to select $h_a=0$. In this case, the agent has a low sensitivity to privacy infringement. A good aspect of this problem formulation is that the constant $\xi_a$ can be intuitively selected based on the meaning of the maximum tolerable conditional probability $\pi_a$.

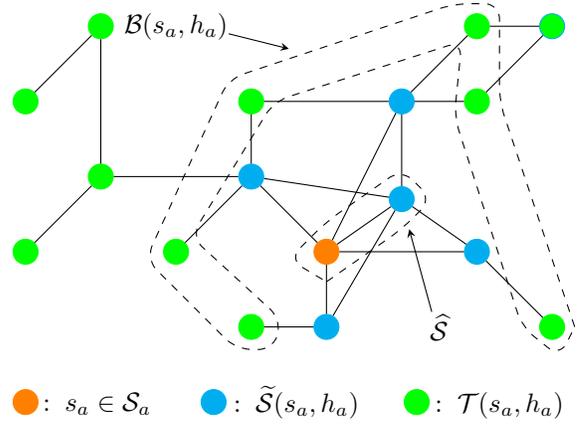
\begin{figure}
\centering
\begin{tikzpicture}
\node[draw,circle,color=cyan  ,minimum size=0.05,fill=cyan  ] (1) at (0,0) {};
\node[draw,circle,color=cyan  ,minimum size=0.05,fill=green ] (2) at (3,4) {};
\node[draw,circle,color=cyan  ,minimum size=0.05,fill=cyan  ] (3) at (2,1) {};
\node[draw,circle,color=orange,minimum size=0.05,fill=orange] (4) at (0,1) {};
\node[draw,circle,color=cyan  ,minimum size=0.05,fill=cyan  ] (5) at (1,3) {};
\node[draw,circle,color=green ,minimum size=0.05,fill=green ] (6) at (3,0) {};
\node[draw,circle,color=green ,minimum size=0.05,fill=green ] (7) at (-4,3) {};
\node[draw,circle,color=green ,minimum size=0.05,fill=green ] (8) at (-3,4) {};
\node[draw,circle,color=cyan  ,minimum size=0.05,fill=cyan  ] (9) at (-1,2) {};
\node[draw,circle,color=green ,minimum size=0.05,fill=green ] (10) at (-3,2) {};
\node[draw,circle,color=green ,minimum size=0.05,fill=green ] (11) at (-1,0) {};
\node[draw,circle,color=green ,minimum size=0.05,fill=green ] (12) at (-4,1) {};
\node[draw,circle,color=green ,minimum size=0.05,fill=green ] (13) at (-1,3) {};
\node[draw,circle,color=green ,minimum size=0.05,fill=green ] (14) at (-2,1) {};
\node[draw,circle,color=green ,minimum size=0.05,fill=green ] (15) at (2,3) {};
\node[draw,circle,color=green ,minimum size=0.05,fill=green ] (16) at (2,4) {};
\node[draw,circle,color=cyan  ,minimum size=0.05,fill=cyan  ] (17) at (1,1.7) {};
\draw[-] (1)  -- (4) ;
\draw[-] (1)  -- (11);
\draw[-] (3)  -- (4) ;
\draw[-] (5)  -- (4) ;
\draw[-] (9)  -- (4) ;
\draw[-] (9)  -- (14);
\draw[-] (9)  -- (10);
\draw[-] (10) -- (12);
\draw[-] (10) -- (8) ;
\draw[-] (8)  -- (7) ;
\draw[-] (5)  -- (13);
\draw[-] (13) -- (9) ;
\draw[-] (3)  -- (6) ;
\draw[-] (15) -- (5) ;
\draw[-] (16) -- (5) ;
\draw[-] (2)  -- (15);
\draw[-] (2)  -- (16);
\draw[-] (17) -- (4) ;
\draw[-] (17) -- (1) ;
\draw[-] (17) -- (3) ;
\draw[-] (17) -- (5) ;
\draw[-] (17) -- (9) ;
\draw[dashed,rounded corners=0.2cm]
						   (+3.3,+0.0)
						-- (+2.3,+3.0)
					    -- (+2.3,+4.3)
					    -- (+1.7,+4.3)
					    -- (-1.3,+3.3)
					    -- (-2.4,+0.9)
					    -- (-1.2,-0.3)
					    -- (-0.7,-0.3)
					    -- (-0.6,+0.1)
					    -- (-1.8,+1.2)
					    -- (-0.7,+3.0)
					    -- (+1.8,+3.8)
					    -- (+1.7,+2.8)
					    -- (+2.6,+0.0)
					    -- (+3.1,-0.4)
					    -- (+3.3,+0.0);
\node[] at (-2.0,+4.0) {$\mathcal{B}(s_a,h_a)$};
\draw[->,>=stealth] (-1.3,+3.9) -- (-0.5,+3.7);
\draw[dashed,rounded corners=0.2cm]
						   (+0.0,+0.6)
						-- (+1.0,+1.3)
						-- (+1.4,+1.7)
						-- (+1.0,+2.1)
						-- (+0.0,+1.4)
						-- (-0.5,+0.9)
						-- (+0.0,+0.6);
\node[] at (+1.5,-0.0) {$\widehat{\mathcal{S}}$};
\draw[->,>=stealth] (+1.4,+0.2) -- (+1.1,+1.3);
\node[draw,circle,color=orange,minimum size=0.05,fill=orange] at (-4,-1) {};
\node[] at (-3.05,-1.02) {$:\;s_a\in\mathcal{S}_a$};
\node[draw,circle,color=cyan ,minimum size=0.05,fill=cyan ] at (-1.5,-1) {};
\node[] at (-0.4,-1.00) {$:\;\widetilde{\mathcal{S}}(s_a,h_a)$};
\node[draw,circle,color=green,minimum size=0.05,fill=green] at (+1.2,-1) {};
\node[] at (+2.3,-1.04) {$:\;\mathcal{T}(s_a,h_a)$};
\end{tikzpicture}
\caption{\label{fig:1} An illustrative example of the sets $\mathcal{S}_a$, $\mathcal{T}(s_a,h_a)$, $\widetilde{\mathcal{S}}(s_a,h_a)$, $\mathcal{B}(s_a,h_a)$, and $\widehat{\mathcal{S}}$ with $h_a=1$.
}
\end{figure}

\section{Privacy-Preserving Policy} \label{sec:sym}
In the light of the aforementioned problem formulation, an external observer that collects the position broadcasts of agent $a$ can form the set of nodes $\mathcal{T}(s_a,h_a)\subseteq \mathcal{V}$ in which agent $a$ has transmitted its position. This is because it is assumed that each agent visits all the nodes in the network. If the observer waits long enough, the set of all the nodes that agent $a$ has visited converges to $\mathcal{T}(s_a,h_a)$ with probability one. Note that
\begin{align}
\mathcal{T}(s_a,h_a)=\{i\in\mathcal{V}\,|\,d(i,s_a)\geq h_a+1\}.
\end{align}
See Fig.~\ref{fig:1} for an illustrative example of this set.
The following immediately follows:
\begin{align}
s_a \in \widetilde{\mathcal{S}}(s_a,h_a):=\mathcal{V}\setminus\mathcal{T}(s_a,h_a).
\end{align}
This is a direct consequence of the selected set of policies.
However, due to the nature of the policy of the agents, the observer can further remove points from $\widetilde{\mathcal{S}}(s_a,h_a)$ that cannot possibly belong to $\mathcal{S}_a$. This is investigated in the remainder of this section.
The boundary of the set $\mathcal{T}(s_a,h_a)$ interfacing with the set $\widetilde{\mathcal{S}}(s_a,h_a)$ can be defined as
\begin{align}
\mathcal{B}(s_a,h_a) = \{i\in\mathcal{T}(s_a,h_a)\,|\,  &\mathcal{D}_i(1)\cap  \widetilde{\mathcal{S}}(s_a,h_a) \neq \emptyset  \},
\end{align}
where, for any $m\in\mathbb{N}$ and any $i\in\mathcal{V}$, the set $\mathcal{D}_i(m)$ denotes the set of nodes $j\in\mathcal{V}$ such that $d(i,j)=m$ (i.e., that are of distance $m$ to node $i$).  Similarly, define
\begin{align}\label{eq:set}
\widehat{\mathcal{S}}
  :=&\bigg\{s \in \widetilde{\mathcal{S}}(s_a,h_a)\,|\,\exists \delta\in\mathbb{N}_{[1,\ell_a+1]}:\mathcal{D}_s(\delta)= \mathcal{B}(s_a,h_a)\nonumber\\
  &\hspace{.6in} \wedge \bigcup_{0\leq \delta'\leq \delta-1}\mathcal{D}_s(\delta')= \widetilde{\mathcal{S}}(s_a,h_a)
  \bigg\}.
\end{align}
where $\ell_a$ is the diameter of the subgraph induced by $\widetilde{\mathcal{S}}(s_a,h_a)$. 
Among all the subsets of $\mathcal{V}$ that can be picked by the observer, $\widehat{\mathcal{S}}$ described by \eqref{eq:set} is the smallest set that is guaranteed to include $s_a$. The following results can be proved.

\begin{theorem} \label{tho:1} $\pi_a=1/|\widehat{\mathcal{S}}|$.
\end{theorem}

\begin{proof} The Bayes' rule (e.g., see~\cite{durrett2010probability}) dictates that
\begin{align}
\mathbb{P}\{s\in\mathcal{S}_a\,|\,\mathcal{T}(s_a,h_a)\}
\propto&\mathbb{P}\{\mathcal{T}(s_a,h_a)\,|\,s\in\mathcal{S}_a\}\mathbb{P}\{s\in\mathcal{S}_a\}\nonumber\\
\propto&\mathbb{P}\{\mathcal{T}(s_a,h_a)\,|\,s\in\mathcal{S}_a\},\label{eqn:prob:1}
\end{align}
where the notation $\propto$ shows that the both sides are proportional to each other (i.e., they are equal to each other if one is multiplied by an appropriate constant). Now, note that
\begin{align*}
\mathbb{P}\{\mathcal{T}(s_a,h_a)\,|\,&s\in\mathcal{S}_a\}\\
&=\mathbbm{1}_{\exists h\in\mathbb{N}:\mathcal{T}(s,h)=\mathcal{T}(s_a,h_a)}\\
&=\mathbbm{1}_{\exists h\in\mathbb{N}:\mathcal{B}(s,h)=\mathcal{B}(s_a,h_a)\wedge\widetilde{\mathcal{S}}(s_a,h_a)=\widetilde{\mathcal{S}}(s,h)},
\end{align*}
where $\mathbbm{1}$ is a characteristic function, i.e., $\mathbbm{1}_p$ is equal to one if the statement $p$ holds true and is equal to zero otherwise. By definition,
$\mathcal{B}(s,h)=\mathcal{D}_s(h+1)$ and $\
\widetilde{\mathcal{S}}(s,h)=\bigcup_{h'\in\mathbb{N}:h'\leq h} \mathcal{D}_s(h').$
Therefore, it can be deduced that
\begin{align} \label{eqn:prob:2}
\mathbb{P}\{\mathcal{T}(s_a,h_a)\,|\,&s\in\mathcal{S}_a\}=\mathbbm{1}_{s\in\widehat{\mathcal{S}}}.
\end{align}
Substituting~\eqref{eqn:prob:2} into~\eqref{eqn:prob:1} gives
\begin{align*}
\mathbb{P}\{&s\in\mathcal{S}_a\,|\,\mathcal{T}(s_a,h_a)\}\propto \mathbbm{1}_{s\in\widehat{\mathcal{S}}}.
\end{align*}
Noting that $\sum_{s\in\mathcal{V}}\mathbb{P}\{s\in\mathcal{S}_a\,|\,\mathcal{T}(s_a,h_a)\}=1$ results in
\begin{align*}
\mathbb{P}\{s\in\mathcal{S}_a\,|\,\mathcal{T}(s_a,h_a)\}=\frac{1}{|\widehat{\mathcal{S}}|}\mathbbm{1}_{s\in\widehat{\mathcal{S}}}.
\end{align*}
This concludes the proof.
\end{proof}

In this case, it can be proved that $\mathcal{E}(s_a,h_a)=\{\{i,j\}\in\mathcal{E}\,|\,i\in \widetilde{\mathcal{S}}(s_a,h_a) \vee j\in \widetilde{\mathcal{S}}(s_a,h_a)\}.$
Hence, Problem~\ref{prob:1} can be cast as
\begin{align*}
\min_{0\leq h_a\leq d_\mathcal{G}} \frac{1}{|\widehat{\mathcal{S}}|}\hspace{-.04in}+\hspace{-.04in}\gamma_a |\hspace{-.02in}\{\hspace{-.02in}\{i,j\}\in\mathcal{E}|i\in \widetilde{\mathcal{S}}(s_a,h_a) \hspace{-.04in}\vee\hspace{-.04in} j\in \widetilde{\mathcal{S}}(s_a,h_a)\hspace{-.02in}\}\hspace{-.02in}|,
\end{align*}
where $d_\mathcal{G}$ denotes the diameter of the graph $\mathcal{G}$. For each $s_a$, this problem can be easily solved by calculating the cost function for all applicable $h_a$ and, subsequently, by selecting $h_a$ corresponding to the smallest value. To calculate the cost, sets $\mathcal{D}_i(\delta)$ for all $i\in\mathcal{V}$ and $\delta$ need to be calculated. These sets can be determined by calculating $\Delta^\delta$ with $\Delta$ denoting the adjacency matrix, i.e., $\Delta_{ij}=1$ if there exists an edge $\{i,j\}\in\mathcal{G}$. All these sets for $\delta$ upto $d_\mathcal{G}$ , can be computed by $\mathcal{O}(|\mathcal{V}|^3d_\mathcal{G})$ operations. Noting that, at worst case $d_\mathcal{G}=\mathcal{O}(|\mathcal{V}|)$, all these sets can be computed by $\mathcal{O}(|\mathcal{V}|^4)$. Therefore, the set $\widehat{\mathcal{S}}$ can be computed by $\mathcal{O}(|\mathcal{V}|^6)$ operations because, at worst case, $|\widetilde{\mathcal{S}}(s_a,h_a)|=\mathcal{O}(|\mathcal{V}|)$. Finally, to perform all these operations for all $h_a$, at most $\mathcal{O}(|\mathcal{V}|^7)$ operations are required. These complexity calculations are very conservative and, for most graphs, much fewer operations are required.

Alternatively, Problem~\ref{prob:2} can be cast as
\begin{align*}
\min_{0\leq h_a\leq d_\mathcal{G}} &|\{\{i,j\}\in\mathcal{E}\,|\,i\in \widetilde{\mathcal{S}}(s_a,h_a) \vee j\in \widetilde{\mathcal{S}}(s_a,h_a)\}|,\\
\mathrm{s.t.}\hspace{.13in}& |\widehat{\mathcal{S}}|\geq 1/\xi_a
\end{align*}
Because $|\{\{i,j\}\in\mathcal{E}\,|\,i\in \widetilde{\mathcal{S}}(s_a,h_a) \vee j\in \widetilde{\mathcal{S}}(s_a,h_a)\}|$ is an increasing function of $h_a$, the optimal solution is to select $h_a$ to be the smallest element of the set $\{h_a\,|\,|\widehat{\mathcal{S}}|\geq 1/\xi_a\}$. This can also be done by checking all $h_a$ with at most $\mathcal{O}(|\mathcal{V}|^7)$ operations.

\section{Asymmetric Policies} \label{sec:asym}
In the previous section, the only viable family of policies for agent $a$ is to stop transmitting its position measurements if its distance in terms of the number of hops to the privacy sensitive node $s_a$ is smaller than or equal to $h_a$. The symmetry of the policy of the agents was utilized by the observer to further reduce the uncertainty  of predicting the privacy sensitive node of the agents. This is clearly a drawback because, for achieving the same level of privacy, an agent that uses a symmetric policy must exclude a larger subgraph from its measurements. However, employing symmetric policies are also advantageous as they can be computed efficiently and their implementation is simpler than the asymmetric ones (the agent only needs to count the number of the nodes on its path). 

To formalize this intuition, the set of applicable policies of the agents is temporarily generalized. Agent $a$ is assumed to follow a policy that instruct its connected device to stop broadcasting its position when it is on a link $\{i,j\}$ if $i\in\mathcal{N}_a$ or $j\in\mathcal{N}_a$. It is assumed that $s_a\in\mathcal{N}_a$. Such an assumption is needed because the observer can deduce that a node is of special importance to an agent, if it stops there for a significant amount of time (as it is not only a way point on the path). The goal is to find the set $\mathcal{N}_a$ optimally. Similar to the previous section, an external observer that collects the position broadcasts of agent $a$ can form the set of nodes $\mathcal{T}(s_a,h_a)\subseteq \mathcal{V}$ in which agent $a$ has transmitted its position. Clearly, $\mathcal{N}_a=\mathcal{V}\setminus\mathcal{T}(s_a,h_a)$. The following result can be proved.

\begin{theorem} \label{tho:2} $\pi_a=1/|\mathcal{N}_a|$.
\end{theorem}

\begin{proof} Similar to the proof of Theorem~\ref{tho:1}, the Bayes' rule can be used to show that
\begin{align*}
\mathbb{P}\{s\in\mathcal{S}_a\,|\,\mathcal{T}(s_a,h_a)\}
\propto&\mathbb{P}\{\mathcal{T}(s_a,h_a)\,|\,s\in\mathcal{S}_a\}\mathbb{P}\{s\in\mathcal{S}_a\}\nonumber\\
\propto&\mathbb{P}\{\mathcal{T}(s_a,h_a)\,|\,s\in\mathcal{S}_a\}.
\end{align*}
Now, note that $\mathbb{P}\{\mathcal{T}(s_a,h_a)\,|\,s\in\mathcal{S}_a\}=\mathbbm{1}_{s\in\mathcal{N}_a}.$
Therefore,
\begin{align*}
\mathbb{P}\{s\in\mathcal{S}_a\,|\,\mathcal{T}(s_a,h_a)\}=\frac{1}{|\mathcal{N}_a|}\mathbbm{1}_{s\in\mathcal{N}_a}.
\end{align*}
This concludes the proof.
\end{proof}

Problem~\ref{prob:1} can then be cast as
\begin{align*}
\min_{\mathcal{N}_a\in 2^{\mathcal{V}}:s_a\in \mathcal{N}_a} \frac{1}{|\mathcal{N}_a|}+\gamma_a |\{\{i,j\}\in\mathcal{E}\,|\,i\in\mathcal{N}_a\vee j\in\mathcal{N}_a\}|.
\end{align*}
Unfortunately, this optimization problem is very hard solve as the number of all the possible sets $\mathcal{N}_a\in 2^{\mathcal{V}}$ such that $s_a\in \mathcal{N}_a$ grows exponentially with the number of the number of the nodes. A similar argument can also be presented for Problem~\ref{prob:2} and is thus removed.

If, in order to reduce the complexity of the problem, the search is conducted on the set of symmetric policies defined to be set of all policies of the form $\mathcal{N}_a=\{s\in\mathcal{V}\,|\,d(s,s_a)\leq h_a\}$ for $h_a\in\mathbb{N}$, the results of the previous section can be recovered. In this case, it can be also proved that
\begin{align*}
\frac{1}{|\mathcal{N}_a|}\leq \frac{1}{|\widehat{\mathcal{S}}|}
\end{align*}
because $\widehat{\mathcal{S}}\subseteq\widetilde{\mathcal{S}}(s_a,h_a)=\mathcal{N}_a$.
Therefore, it can be inferred that the symmetric polices minimize an upper-bound on the costs over a smaller set of policies, i.e., the set of asymmetrical policies. This results in their superior efficiency in terms of computation and implementation.

\section{Varying Density} \label{sec:density}
In most urban areas, the density of the population is not homogeneous throughout the city. Let $\rho:\mathcal{V}\rightarrow\mathbb{R}_{\geq 0}$ be a mapping that determines the density around any node. Following this observation, the observer can no longer assume a uniform prior on the {  private nodes}. Therefore, $\mathbb{P}\{s\in\mathcal{S}_a\}=\rho(s)/\sum_{v\in\mathcal{V}}\rho(v)$. This captures the fact that it is more likely that the privacy sensitive node $s_a$ of agent~$a$ to belong to a densely populated area. In this case, the following result can be proved.

\begin{theorem} \label{tho:3} $\pi_a=\rho(s)/\sum_{v\in\widehat{\mathcal{S}}}\rho(v)$.
\end{theorem}

\begin{proof} Similar to the proof of Theorem~\ref{tho:1}, it can be shown that
\begin{align*}
\mathbb{P}\{s\in\mathcal{S}_a\,|\,\mathcal{T}(s_a,h_a)\}
\propto&\mathbb{P}\{\mathcal{T}(s_a,h_a)\,|\,s\in\mathcal{S}_a\}\mathbb{P}\{s\in\mathcal{S}_a\}\nonumber\\
\propto&\mathbb{P}\{\mathcal{T}(s_a,h_a)\,|\,s\in\mathcal{S}_a\}\rho(s)\\
=&\mathbbm{1}_{s\in\widehat{\mathcal{S}}}\rho(s).
\end{align*}
As a result,
\begin{align*}
\mathbb{P}\{s\in\mathcal{S}_a\,|\,\mathcal{T}(s_a,h_a)\}=\frac{\rho(s)}{\sum_{v\in\widehat{\mathcal{S}}}\rho(v)}\mathbbm{1}_{s\in\widehat{\mathcal{S}}}.
\end{align*}
This concludes the proof.
\end{proof}

Following the result of Theorem~\ref{tho:3}, the problems that the agents need to solve can be adapted and the algorithm in Section~\ref{sec:sym} can be used to find an optimal privacy-preserving policy.

Note that a ``sensible'' agent should also put less emphasize on protecting its privacy if its privacy sensitive node is in a densely populated area. This is because, in densely populated areas, even if the observer pinpoints the location of an agent upto a node, there are many residential and commercial areas in the surrounding that makes identifying the physical location of the agent impossible. Such a behaviour can be reflected in the selection of the term $\gamma_a$ (if the agent makes such a decision).

\section{Numerical Example} \label{sec:numerical}
First, the results of Section~\ref{sec:sym} are demonstrated on random graphs. Consider a graph with $|\mathcal{V}|=1000$ nodes. The positions of the nodes are distributed uniformly inside a unit rectangle $[0,1]\times[0,1]$. If the Euclidean distance between two nodes is smaller than or equal to $0.1$, the nodes are assumed to be connected. This creates an undirected graph, which is referred to as a random geometric graph; see~\cite{penrose2003random} for more information.

\begin{figure}
  \centering
  \includegraphics[width=0.9\linewidth]{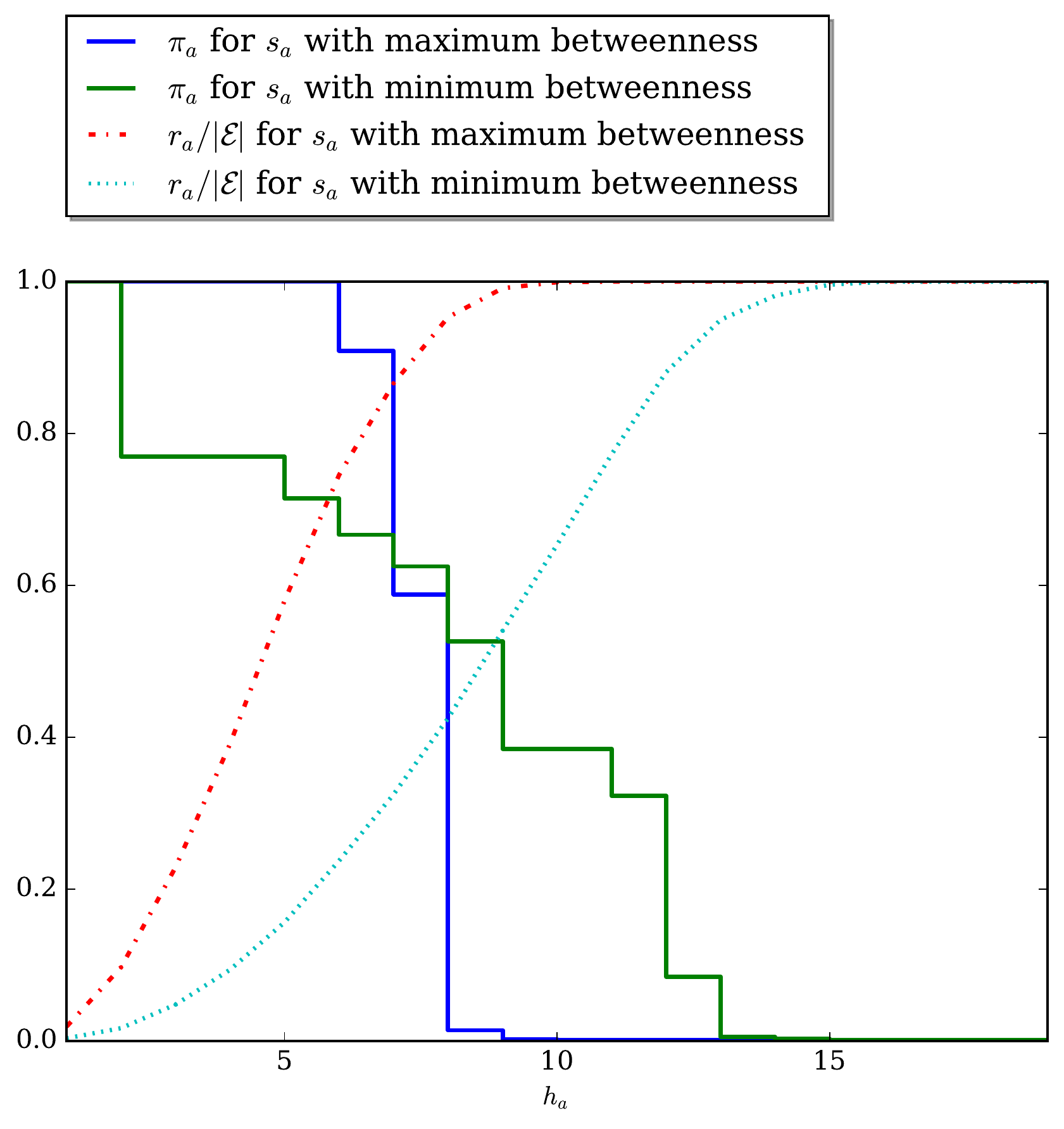}
  \caption{$\pi_a$ and $r_a$ as a function of $h_a$ for nodes with the maximum and the minimum betweenness averaged over ten random geometric graphs with $1000$ nodes in unit rectangle and connectivity radius of $0.1$.}
  \label{fig:test}
\end{figure}

The values of $\pi_a$ and $r_a$ (as function of $h_a$) are only illustrated for the nodes with the maximum and the minimum betweenness measures (as two extreme behaviours). The betweenness measure of centrality for a node is defined as the number of the shortest paths in the graph that pass through that node divided by the number of all the shortest paths in the graph. Intuitively, a node with a relatively large betweenness measure connect many nodes to each other. However, a node  with a small betweenness measure can be removed from the graph with no devastating consequences. Fig.~\ref{fig:test} shows $\pi_a$ and $r_a$ as a function of $h_a$ for nodes with the maximum and the minimum betweenness averaged over ten random geometric graphs. Based on the illustrated numerical example, it can be seen that, for the node with the maximum betweenness, a fairly large $h_a$ should be selected to reduce $\pi_a$ initially. This is because for these nodes the set $\mathcal{B}(s_a,h_a)$ is very large and thus $\widehat{\mathcal{S}}$ most often contains only a few nodes. In addition, by increasing $h_a$, $r_a$ rapidly increases. This is because the central nodes are often very close to all the nodes (as many shortest paths go through them). The reverse behaviour is observed for the node with the minimum betweenness.

Transportation networks are most often highly structured graphs as their design is not random (or at least one hopes) and follows careful consideration by the local governments. Therefore, in the next simulation, a scenario in the City of Melbourne. The goal is to ascertain the achievable levels of privacy for an agent starting his travel from the dark blue node\footnote{The interest in this node is not arbitrary as it pinpoints the position of a former residence of one of the authors.} in Fig.~\ref{fig:melb}. The possible levels of privacy infringement are depicted in Figure~\ref{fig:iman_privacy}. Evidently, with a fairly low $h_a$, $p_a$ can be reduced drastically. This is also achieved at a relatively low cost because $r_a$ is still small.

\begin{figure}
  \centering
  \includegraphics[width=0.9\linewidth]{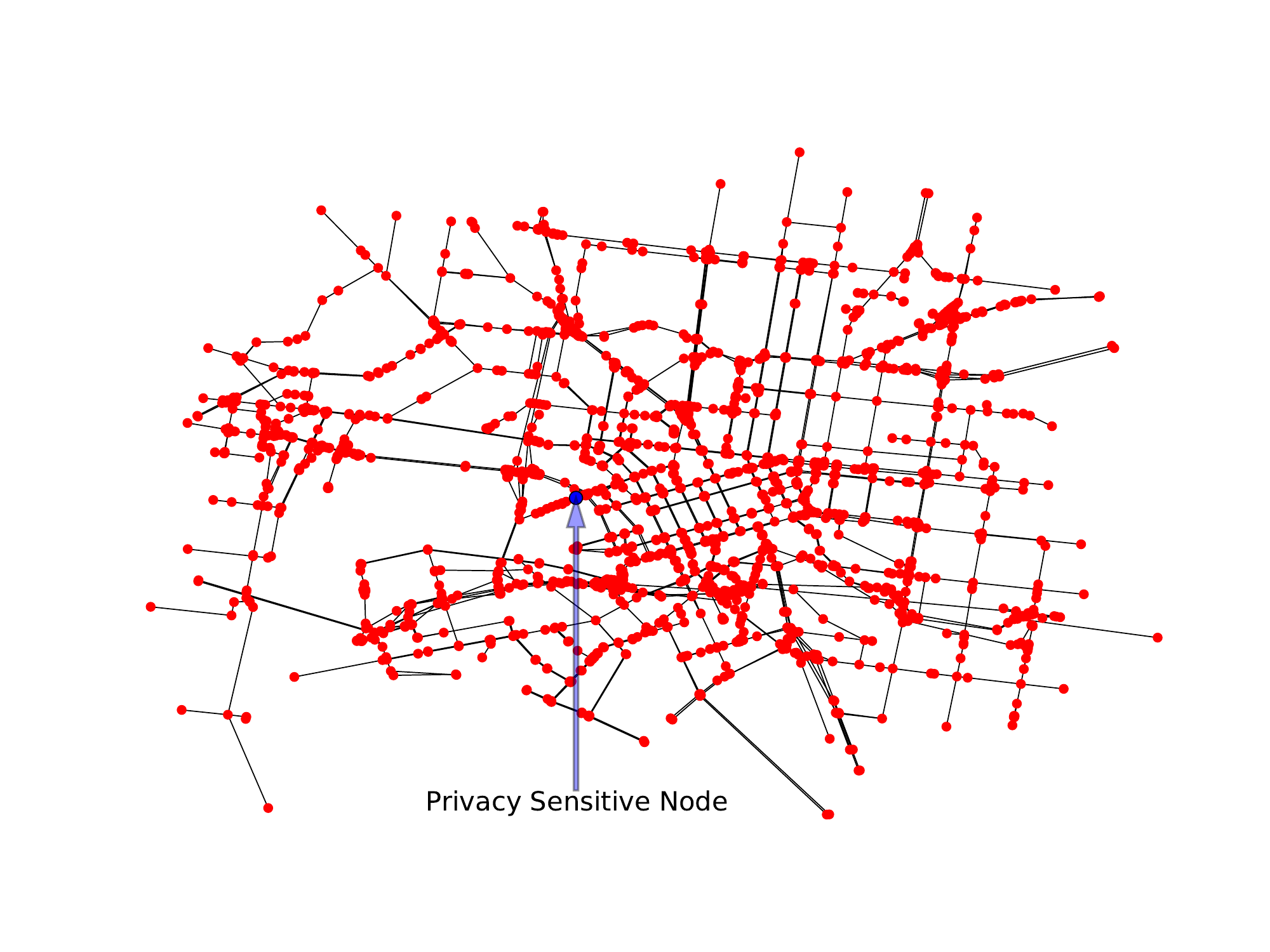}
  \caption{The roads in the City of Melbourne}
  \label{fig:melb}
\end{figure}
\begin{figure}
  \centering
  \includegraphics[width=0.9\linewidth]{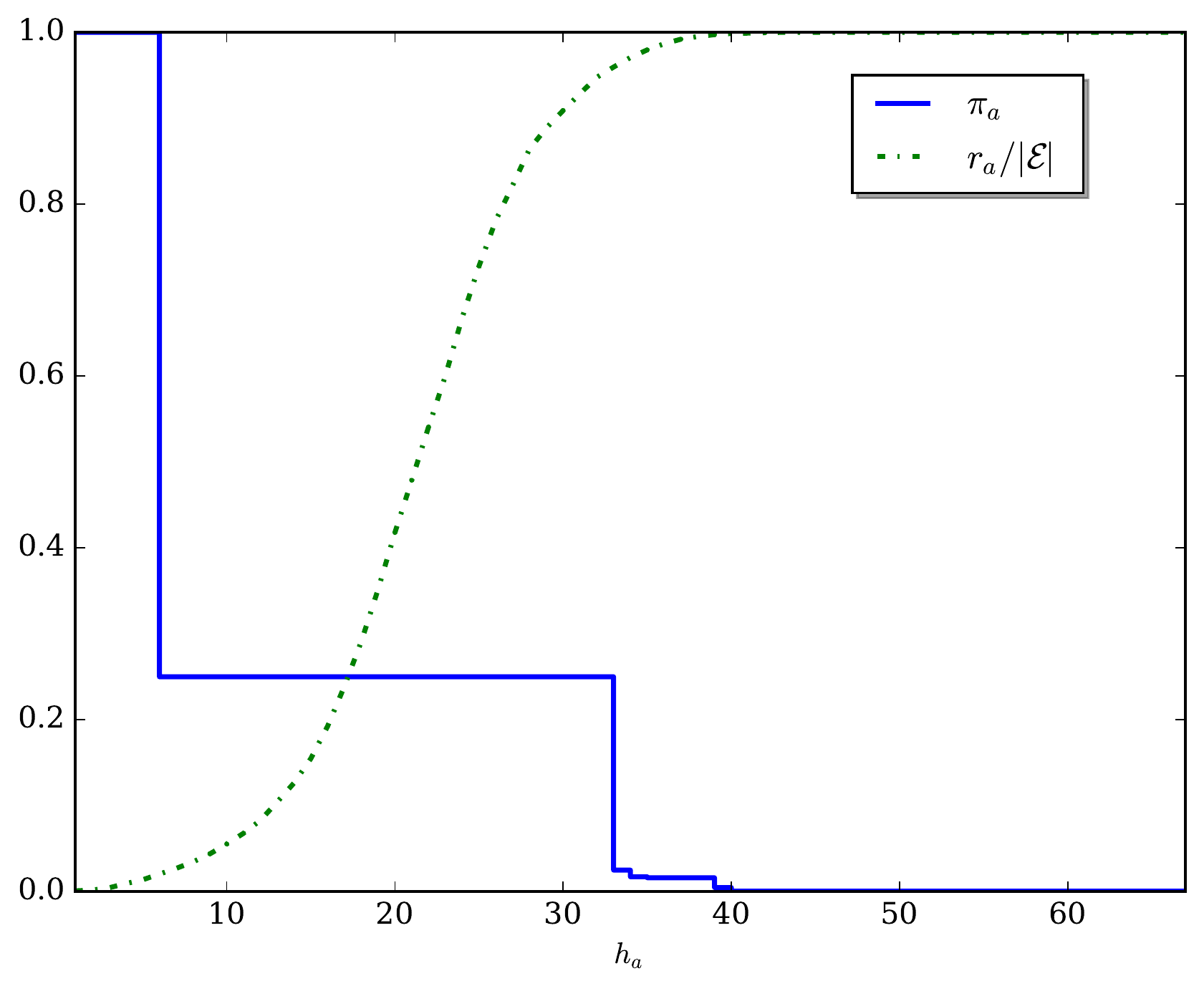}
  \caption{Different values of $\pi_a$ versus the ratio of the nodes where the location is not reported for different values of $h_a$.}
  \label{fig:iman_privacy}
\end{figure}

\section{Conclusions and Future Work} \label{sec:conclusions}
An algorithm for finding an optimal policy for preserving the privacy of the agents in a participatory-sensing scheme for traffic estimation is presented. The effect of the betweenness measure of centrality on the policy of the agents is numerically explored. In can be seen that, for the node with the maximum betweenness, a fairly large exclusion radius should be selected to enhance the privacy. In addition, by slightly increasing radius, the number of the nodes over which the agent stop its measurement transmission rapidly increases. The reverse behaviour is seen for the node with the minimum betweenness. {  Finally, note that if the agents want to hide their position on some edges (denoted by the \textit{privacy-sensitive edges}), the same results can be used. To do so, the definition of the graph representing the transportation system should be modified. Specifically, to find the optimal policy, the line graph (see~\cite[pp.\,71--82]{harary1969graph}) of the transportation network should be constructed. A line graph of a graph $\mathcal{G}=(\mathcal{V},\mathcal{E})$ is an undirected graph $L(\mathcal{G})$ with the vertex set $\mathcal{V}_{L(\mathcal{G})}=\mathcal{E}$ and the edge set
$\mathcal{E}_L(\mathcal{G})=\{(e,\bar{e})\in\mathcal{V}_{L(\mathcal{G})}\times \mathcal{V}_{L(\mathcal{G})}\,|\,i=\bar{i} \vee j=\bar{j}
 \mbox{ for } e=\{i,j\}, \bar{e}=\{\bar{i},\bar{j}\} \}.$ In the line graph of the transportation, the roads represent the nodes and two roads are connected to each other by an edge (in the line graph) if they intersect.}

The future work can focus on understanding the effect of determined policies using simulators for transportation systems to measure the quality of the participatory-sensing schemes. { Also financial incentives can be designed to improve the estimation quality in different areas of the city.}

\bibliographystyle{IEEEtran}
\bibliography{ref}
\end{document}